\documentclass[a4paper,11pt]{amsart}

\usepackage[a4paper]{geometry}
\usepackage{amsmath,amsthm}
\usepackage{amssymb,amsfonts}
\usepackage{hyperref}
\usepackage{tikz}
\usepackage{enumerate}
\usepackage{graphicx}
\usepackage{calligra}
\usepackage{stmaryrd}
\usepackage{chngpage}
\DeclareSymbolFont{bbold}{U}{bbold}{m}{n}
\DeclareSymbolFontAlphabet{\mathbbm}{bbold}

\title{Obstructions to lifting abelian subalgebras of corona algebras}

\theoremstyle{plain}
	\newtheorem{theorem*}{Theorem}
	\newtheorem{theorem}{Theorem}[section]
	\newtheorem{proposition}[theorem]{Proposition}
	\newtheorem{lemma}[theorem]{Lemma}
	\newtheorem{corollary}[theorem]{Corollary}

\theoremstyle{definition}

	\newtheorem{example}[theorem]{Example}
	
	\newtheorem{question}[theorem]{Question}
	\newtheorem*{acknowledgements}{Acknowledgements}

\newcommand{\C}{\mathbb{C}}
\newcommand{\N}{\mathbb{N}}

\newcommand{\set}[1]{\{#1\}}					
\newcommand{\cstar}{$\mathrm{C}^\ast$}
\renewcommand{\phi}{\varphi} 

\newcommand{\Addresses}{{
  \bigskip
  \footnotesize

   A.~Vaccaro, \textsc{Department of Mathematics, University of Pisa, Largo Bruno Pontecorvo 5
    Pisa, Italy, 56127 - Department of Mathematics and Statistics, York University, 4700 Keelee Street,
    North York, Ontario, Canada, M3J 1P3}\par\nopagebreak
  \textit{E-mail address}: \texttt{vaccaro@mail.dm.unipi.it}\par\nopagebreak
  \textit{URL}: \texttt{http://people.dm.unipi.it/vaccaro/index.html}
  }}

\author{Andrea Vaccaro}
\date{}

\begin{document}
\maketitle
\begin{abstract}
Let $A$ be a non-commutative, non-unital \cstar-algebra.
Given a set of commuting positive elements in the corona algebra $Q(A)$, we study
some obstructions to the existence of a
commutative lifting of such set to the multiplier algebra $M(A)$.
Our focus are the obstructions caused by the size of the collection we want to lift. 
It is known that no obstacles show up when lifting a countable family of commuting projections,
or of pairwise orthogonal positive elements. However, this is not
the case for larger collections. We prove in fact that
for every primitive, non-unital, $\sigma$-unital \cstar-algebra $A$,
there exists an uncountable set of pairwise orthogonal positive elements in $Q(A)$
such that no uncountable subset of it can be lifted
to a set of commuting elements of $M(A)$. Moreover, the positive elements in $Q(A)$
can be chosen to be projections if $A$ has real rank zero.
\end{abstract}

\section{Introduction}
Let $A$ be a non-unital \cstar-algebra,
denote its multiplier algebra by $M(A)$,
its corona algebra (namely $M(A)/A$) by $Q(A)$, and the quotient
map from $M(A)$ onto $Q(A)$ by $\pi$.
A \emph{lifting} in $M(A)$ of a set $B \subseteq Q(A)$ is a set
$C \subseteq M(A)$ such that $\pi[C]=B$. The study of which
properties of $B\subseteq Q(A)$ can be preserved in
a lifting, as well as the analysis of the relations between $B$ and its preimage
$\pi^{-1}[B]$, has produced a rich theory with strong connections
to the study of stable relations in \cstar-algebras.
A general introduction to this subject can be found in \cite{loring}.

This note focuses on liftings of abelian subalgebras of corona algebras.
This topic has been widely studied, for instance, as a mean
to produce interesting examples of $*$-algebras, and
in the investigation of the masas (maximal abelian subalgebras)
of the Calkin algebra $Q(H)$.
In \cite{akem}, for example,
the authors produce, by means of a lifting,
a nonseparable \cstar-algebra whose abelian subalgebras are all separable.
Their proof assumes the Continuum Hypothesis, which was later shown to be not necessary
(see \cite[Corollary 6.7]{popa} and also \cite{bice}). Another application of the Continuum Hypothesis
to liftings of abelian subalgebras of corona algebras can be found in \cite{anderson}.
Here the author builds a masa of $Q(H)$ which is generated by its projections
and does not lift to a masa in $B(H)$. In this case it is not known whether the Continuum
Hypothesis can be dropped (see \cite{shelah2011masas}).
More recently, the study of liftings led to the first example of an amenable nonseparable Banach algebra which is not isomorphic to a \cstar-algebra (see \cite{farah};
see also \cite{vigna}).

In this paper we focus on the following problem.
Let $A$ be a non-commutative, non-unital \cstar-algebra,
and let $B$ be a commutative family in $Q(A)$.
What kind of obstructions could prevent the existence of a commutative lifting of $B$ in $M(A)$?
We consider collections with
various properties, but our main concern and focus is the role played by the cardinality
of the set that we want to lift. The following table summarizes all the cases that we are going
to analyze. The symbols
``$\checkmark$'' and ``$\times$'' indicate whether it is possible or not to have a lifting
for collections on the left column whose size
is the cardinal in the top line.
\vspace{0.1cm}
\begin{adjustwidth}{-2cm}{-2cm}
\begin{center}
 \begin{tabular}{| c | c | c | c |}
    \hline
    $Q(A) \to M(A)$ & $<\aleph_0$ & $ \aleph_0$ & $ \aleph_1$ \\ \hline
    Commuting self-adjoint $\to$ Commuting self-adjoint & $\times$& $\times$ & $\times$ \\ \hline
    Commuting projections $\to$ Commuting projections & $\checkmark$ in $Q(H)$ & $\checkmark$ in $Q(H)$ & $\times$ \\ \hline
    Commuting projections $\to$ Commuting positive & $\checkmark$& $\checkmark$ & $\times$ \\ \hline
    Orthogonal positive $\to$ Orthogonal positive& $\checkmark$ & $\checkmark$ & $\times$ \\ \hline
    Orthogonal positive $\to$ Commuting positive & $\checkmark$& $\checkmark$ & $\times$ \\ \hline
        \end{tabular}
        \end{center}
        \end{adjustwidth}
        \vspace{0.2cm}
It is clear from the table that starting with an uncountable collection
is a fatal obstruction.
We also remark that the two columns in the middle, representing the lifting problem for finite and
countable collections, have the same values. One reason for this phenomenon is that the only obstructions in this scenario are of K-theoretic nature and involve only a finite number
of elements, as we shall see in the next paragraph (see also \cite{ken2}).
This situation also relates to other compactness phenomena (at least at the countable level)
that corona algebras of
$\sigma$-unital algebras satisfy, due to their partial countable saturation (see \cite{satur}).
Most of the results in the table about finite and countable families
are already known (\cite{loring}, \cite[Lemma 5.34]{setoperator}, \cite[Lemma 10.1.12]{loring}).
The main contribution of this paper concerns the right column,
for which some theorems about projections in the Calkin algebra have already been proved
(\cite[Theorem 5.35]{setoperator}, \cite{bice}).

Let $A$ be $K(H)$, the algebra of the compact operators on a separable Hilbert
space $H$, so that $M(A)=B(H)$ and $Q(A)=Q(H)$.
By a well-known K-theoretic obstruction, the unilateral shift is a normal element
in $Q(H)$ which does not lift to a normal element in $B(H)$
(more on this in \cite{bdf} and \cite{david}).
An element is normal if and only if its real and imaginary part commute. This proves
that it is not always possible to lift a couple of commuting self-adjoint
elements in a corona algebra to commuting self-adjoint elements in the multiplier algebra.

In order to bypass this obstruction, we add some conditions to the collection which we
want to lift. 
In \cite[Lemma $5.34$]{setoperator} it is proved that any countable family of
commuting projections in the Calkin algebra can be lifted to a family
of commuting projections in $B(H)$. Moreover, the authors provide a lifting of
simultaneously diagonalizable projections. Proving a more general statement about liftings,
in Section \ref{sctn:1} we show
that any countable collection of commuting projections in a corona algebra can be lifted
to a commutative family of positive elements in the
multiplier algebra\footnote{We remark that it is not always possible to lift
projections in a corona algebra to projections in the multiplier algebra. 
Such lifting is not possible for instance when $Q(A)$ has real rank zero but
$M(A)$ has not, which is the case for $A=Q(H) \otimes K(H)$ (see
\cite[Example 2.7(iii)]{zhang})
or $A=\mathcal{Z} \otimes K(H)$, where $\mathcal{Z}$ is the
Jiang-Su algebra (see \cite{linng}).}.

Two elements in a \cstar-algebra are \emph{orthogonal} if their product is zero.
Any countable family of orthogonal positive elements in a corona algebra admits a commutative lifting.
This is a consequence of the more general result \cite[Lemma 10.1.12]{loring}, which is relaid
in this paper as Proposition \ref{prop}.

In general, we cannot expect to be able to generalize verbatim the above result
for uncountable families of orthogonal positive elements.
This is the case since, by a cardinality obstruction, a multiplier
algebra $M(A)$ which can be faithfully represented on a separable Hilbert space $H$,
cannot contain an uncountable collection of orthogonal positive elements. The existence of such
a collection in $M(A)$
(and thus in $B(H)$) would in fact imply the existence
of an uncountable set of orthogonal vectors in $H$, contradicting the separability of $H$.

We could still ask whether it is possible
to lift an uncountable family of orthogonal positive elements
to a family of commuting positive elements.
This leads to an obstruction of set-theoretic nature. In Theorem $5.35$ of
\cite{setoperator}, it is shown that there exists
an $\aleph_1$-sized collection of orthogonal projections in the Calkin algebra whose
uncountable subsets cannot be lifted
to families of simultaneously diagonalizable projections in $B(H)$. This result is refined
in Theorem $7$ of \cite{bice}, where the authors provide
an $\aleph_1$-sized set of orthogonal projections in $Q(H)$ which
contains no uncountable subset that lifts to a collection of commuting operators in $B(H)$.
The main result of this paper is a generalization of Theorem 7 in \cite{bice}.
A \cstar-algebra is $\sigma$\emph{-unital}
if it has a countable approximate unit, and it is
\emph{primitive} if it admits a faithful irreducible representation.

\begin{theorem} \label{thrm}
Assume $A$ is a primitive, non-unital, $\sigma$-unital \cstar-algebra.
Then there is a collection of $\aleph_1$ pairwise
orthogonal positive elements of $Q(A)$ containing no uncountable
subset that simultaneously lifts to commuting elements in $M(A)$.
\end{theorem}

\begin{corollary} \label{crlry}
Assume $A$ is a primitive, real rank zero, non-unital, $\sigma$-unital \cstar-algebra.
Then there is a collection of $\aleph_1$ pairwise
orthogonal projections of $Q(A)$ containing no uncountable
subset that simultaneously lifts to commuting elements in $M(A)$.
\end{corollary}

The proof of Theorem \ref{thrm} is inspired by the combinatorics used in
\cite{bice} and \cite{setoperator}, which goes back to Luzin and
Hausdorff, and to the study of uncountable
almost disjoint families of subsets of $\N$ and Luzin's families (see \cite{luzin}).
We remark that no additional set theoretic assumption (such as the Continuum Hypothesis)
is required in our proof.

The paper is structured as follows: in Section \ref{sctn:1} we outline the results
needed to settle the problem of liftings of countable families of commuting projections and
of orthogonal positive elements. Section \ref{sctn:2} is devoted to the
proof of Theorem \ref{thrm}, while concluding remarks and questions can
be found in Section \ref{sctn:3}.

\section{Countable collections} \label{sctn:1}
Denote the set of self-adjoint and of positive elements of a \cstar-algebra $A$
by $A_{sa}$ and $A_+$, respectively. Given a compact Hausdorff space $X$, $C(X)$ is the
\cstar-algebra of the continuous functions from $X$ into $\C$.

In \cite[Lemma $5.34$]{setoperator} Farah and Wofsey
prove that any countable set of commuting projections in the Calkin algebra can
be lifted to a set of simultaneously diagonalizable projections in $B(H)$.
The thesis of the following proposition is weaker, but it holds in a more general
context.
\begin{proposition}
Let $\phi:A \to B$ be a surjective $*$-homomorphism between two \cstar-algebras and
let $\set{p_n}_{n\in \N}$ be a collection of commuting projections of $B$.
Then there exists a set $\set{q_n}_{n\in \N}$ of commuting positive elements of $A$
such that $\phi(q_n)=p_n$.
\end{proposition}
\begin{proof}
We can assume that both $A$ and $B$ are unital, that $\phi(1_A)=1_B$ and that
$1_B \in \set{p_n}_{n\in \N}$.
Let $C\subseteq B$ be the abelian \cstar-algebra generated by the set
$\set{p_n}_{n\in \N}$. Consider the element
\[
b=\sum_{n\in \N} \frac{2p_n -1}{3^n}.
\]
Let $X$ be the spectrum of $b$ in $A$. The algebra $C$ is generated by $b$
(see \cite[p. 293]{rickart} for a proof), thus $C\cong C(X)$.
Fix $a \in A$ such that $\phi(a) = b$. The element $(a +a^*)/2$ is still in the preimage of $b$
since $b$ is self-adjoint, thus we can assume $a \in A_{sa}$.
If $Y$ is the spectrum of $a$, we have in general
that $X\subseteq Y$. Fix $f_n \in C(X)_+$ such that $f_n(b)=p_n$. Since
the range of $f_n$ is contained in $[0,1]$ and the spaces $Y$ and $X$ are compact and Hausdorff,
by the Tietze Extension Theorem (\cite[Theorem 15.8]{will}), for every $n \in \N$,
there is a continuous $F_n: Y \to [0,1]$
such that $F_n\restriction_X=f_n$. Set $q_n=F_n(a)$. The map $\phi$ acts on $C(Y)$ as
the restriction on $X$ (here we identify $\mathrm{C}^\ast(a)$
and $\mathrm{C}^\ast(b)$ with $C(Y)$ and $C(X)$ respectively), therefore $\phi(q_n)
= p_n$ for every $n \in \N$.
\end{proof}
The $q_n$'s can be chosen to be projections if there is a self-adjoint $a$ in the preimage of $b$
whose spectrum is $X$. By the Weyl-von Nuemann theorem,
this is the case when $\phi$ is the quotient map from $B(H)$ onto the
Calkin algebra (see \cite[Theorem II.4.4]{kend}).

We focus now on lifting sets of positive orthogonal elements, starting with
a set of size two.
Let therefore $\phi:A \to B$ be a surjective $*$-homomorphism
of \cstar-algebras, and let $b_1,b_2\in B_+$ be such that $b_1b_2=0$.
Consider the self-adjoint $b=b_1-b_2$
and let $a\in A_{sa}$ be such that $\phi(a)=b$. The positive and the negative part of $a$ are
two orthogonal positive elements of $A$ such that $\phi(a_+)=b_1$, $\phi(a_-)=b_2$.
The situation is analogous when dealing
with countable collections, as shown in Lemma $10.1.12$ of \cite{loring}. 
\begin{proposition}[{\cite[Lemma 10.1.12]{loring}}] 	\label{prop}
Assume $\phi:A \to B$ is a surjective $*$-homomorphism between
two \cstar-algebras.
Let $\set{b_n}_{n\in \N}$ be a collection of orthogonal positive elements in $B$. Then there exists
a set $\set{a_n}_{n\in \N}$ of orthogonal positive elements in $A$ such that $\phi(a_n)=b_n$.
\end{proposition}

\section{Uncountable collections} \label{sctn:2}
Throughout this section, let $A$ be a primitive, non-unital,
$\sigma$-unital \cstar-algebra. We can thus assume that $A$ is a non-commutative
strongly dense \cstar-subalgebra of $B(H)$ for a certain Hilbert space $H$.
A sequence of operators $\set{x_n}_{n\in \N}$ \emph{strictly converges} to $x\in B(H)$ if and only if
$x_n a \to xa$ and $ a x_n \to a x$ in norm for all $a \in A$.
In this scenario $M(A)$ can be identified with the idealizer
\[
\set{x \in B(H) : xA\subseteq A , Ax\subseteq A}
\]
or with the strict closure of $A$ in $B(H)$.
Given two elements $a,b$ in a \cstar-algebra $A$, we denote the commutator
$ab - ba$ by $[a,b]$.
From now on, let $(e_n)_{n\in \N}$ be an approximate unit of $A$
such that:
\begin{enumerate}
\item $e_0 = 0$;
\item $\lVert e_i - e_j \rVert =1$ for $i\not= j$;
\item $e_i e_j=e_i$ for every $i<j$.
\end{enumerate}
Such an approximate unit exists since $A$ is $\sigma$-unital,
as proved in Section $2$ of \cite{pedersencorona}.

The proof of Theorem \ref{thrm} follows closely
the one given by Bice and Koszmider for \cite[Theorem 7]{bice}, and a lemma
similar to \cite[Lemma 6]{bice} is required.

\begin{lemma} \label{lemma}
Let $A$ be a primitive, non-unital, $\sigma$-unital \cstar-algebra.
There exists a family $(a_{\beta})_{\beta \in \aleph_1}
\subseteq M(A)_+\setminus A$
such that:
\begin{enumerate}
\item $\lVert a_{\beta} \rVert = 1$ for all $\beta \in \aleph_1$;
\item $a_{\alpha} a_{\beta} \in A$ for all distinct $\alpha ,\beta \in \aleph_1$;
\item given $d_1, d_2 \in M(A)$, for all $\beta \in \aleph_1$, all $n \in \N$,
and all but finitely many
$\alpha < \beta$:
\[
\lVert 
[(a_{\alpha} + d_1e_n),(a_{\beta} +d_2e_n)] \rVert \ge\frac{1}{8}.
\]
\end{enumerate}
\end{lemma}
The rough idea to prove this lemma is
to build, for every $\beta < \aleph_1$, a strictly increasing function
$f_{\beta}: \N \to \N$ and a norm-bounded sequence
$\set{c^\beta_k}_{k \in \N}\subseteq A_+$ to define
\[
a_{\beta}= \sum_{k\in \N} (e_{f_{\beta}(2k+1)}-e_{f_{\beta}(2k)})^{\frac{1}{2}}
c^{\beta}_k (e_{f_{\beta}(2k+1)}-e_{f_{\beta}(2k)})^{\frac{1}{2}}.
\]
Note that this series belongs to $M(A)$ by Theorem $4.1$ in 
\cite{pedersencorona} (see also \cite[Item (10) p.48]{satur}). In order to satisfy the thesis of the lemma, we will build each $c^\beta_k$
so that, for some $\alpha < \beta$ and some $n\in \N$, the following holds
\[
\lVert 
[(a_{\alpha} + e_n),(c^{\beta}_k +e_n)]
 \rVert \ge \frac{1}{8}.
\]
The choice of $f_{\beta}$ will guarantee orthogonality in $Q(A)$
exploiting, for $n_2<n_1<m_2<m_1$, the following fact:
\[
(e_{m_1}-e_{m_2})(e_{n_1}-e_{n_2})=0.
\]
The main ingredient used to build $c^{\beta}_k$ is Kadison's Transitivity Theorem,
which we are allowed to use since $A$ is primitive.

\begin{proof}[Proof of Lemma \ref{lemma}]
Since the \cstar-algebra $A$ is primitive, we can assume that there is a Hilbert space
$H$ such that $A \subseteq B(H)$ and $A$ acts irreducibly on $H$.
For each $n<m$, denote the space $\overline{(e_m-e_n)H}$ by $S_{n,m}$.
We start by building $a_0$. Let $f:\N \to \N$ be defined as follows:
\[
f(n)= 
 \begin{cases}
 	 2^{n+1}-1 & \text{ if }n \text{ is even}\\
      2^n & \text{ if }n \text{ is odd.}
\end{cases} 
\]
For every $k\in \N$ there is a unit
vector $\xi$ in the range of $e_{f(2k+1)} - e_{f(2k)}$. By the definition of the approximate unit
$(e_n)_{n \in \N}$, the vector $\xi$ is a 1-eigenvector of $e_{f(2k+2)}$.
This, along with the (algebraic) irreducibility of $A \subseteq B(H)$, entails that
\[
A S_{f(2k+1),f(2k)}=H.
\]
Denote the algebra $\overline{(e_{f(2k+1)} - e_{f(2k)})A(e_{f(2k+1)} - e_{f(2k)})}$ by $A_k$.
We have that
\[
A_k H\supseteq S_{f(2k),f(2k+1)}.
\]
Let $\xi^0_k,\eta^0_k \in S_{f(2k),f(2k+1)}$ be two orthogonal\footnote{
We can always assume $S_{n,n+1}$ has at least $2$ linearly independent vectors
for each $n\in \N$ by taking, if necessary, a subsequence $(e_{k_j})_{j\in \N}$ from the original
approximate unit.} norm one vectors.
Since $A$ acts irreducibly on $H$ and $A_k$ is a hereditary subalgebra of $A$, it follows that
$A_k$ acts irreducibly on $B(A_kH)$
(see \cite[Theorem 5.5.2]{murphy}). Therefore, by Kadison's Transitivity Theorem,
we can find a self-adjoint $c^0_k \in A_k$ such that
\[
c^0_k(\xi^0_k)=\xi^0_k,
\]
\[
c^0_k(\eta^0_k)=0,
\]
and
$ \lVert c^0_k \rVert = 1$. We can suppose that $c^0_k$ is positive by taking its square,
doing so will not change its norm nor the image of $\xi^0_k$ and $\eta^0_k$.
Consider the function
\[
f_0(n)= 
 \begin{cases}
  f(n)-1 & \text{ if }n \text{ is even} \\
      f(n)+1 & \text{ if }n \text{ is odd.} 
     
\end{cases} 
\]
We have that
\[
e_{f_0(2k+1)}c^0_k=c^0_ke_{f_0(2k+1)}=c^0_k,
\]
\[
e_{f_0(2k)} c^0_k = c^0_k e_{f_0(2k)}=0.
\]
This entails
\[
(e_{f_0(2k+1)}-e_{f_0(2k)}) c^0_k= c^0_k = c^0_k
(e_{f_0(2k+1)}-e_{f_0(2k)})
\]
and therefore also
\[
c^0_k=(e_{f_0(2k+1)}-e_{f_0(2k)})^{1/2} c^0_k
(e_{f_0(2k+1)}-e_{f_0(2k)})^{1/2}.
\]
The norm $\lVert c^0_k \rVert$ is bounded by 1 for every $k \in \N$, therefore the sum
\[
a_0=\sum_{k\in \N} c^0_k= \sum_{k\in \N} (e_{f_0(2k+1)}-e_{f_0(2k)})^{1/2} c^0_k
(e_{f_0(2k+1)}-e_{f_0(2k)})^{1/2}
\]
is strictly convergent (see \cite[Theorem 4.1]{pedersencorona} or \cite[Item (10) p.48]{satur}),
hence $a_0\in M(A)_+$. Furthermore:
\begin{align*}
\lVert a_0 \rVert &= \lVert \sum_{k\in \N} (e_{f_0(2k+1)}-e_{f_0(2k)})^{1/2} c^0_k
(e_{f_0(2k+1)}-e_{f_0(2k)})^{1/2} \rVert \le \\
&\le  \lVert \sum_{k\in \N} e_{f_0(2k+1)}-e_{f_0(2k)} \rVert \le 1.
\end{align*}
In order to show that $a_0 \notin A$,
first observe that
\[
a_0(\xi^0_k)= \sum_{m< k} c^0_m(\xi^0_k) + c^0_k(\xi^0_k) +
\sum_{m>k}c^0_m (\xi^0_k)=c^0_k(\xi^0_k) = \xi^0_k.
\]
The first sum annihilates since $\xi^0_k \in S_{f(2k),f(2k+1)}$ implies
$\xi^0_k = (e_{f_0(2k+1)} - e_{f_0(2k)})(\xi^0_k)$, and for $m<k$
\[
c^0_m(e_{f_0(2k+1)} - e_{f_0(2k)})(\xi^0_k)=
c^0_m e_{f_0(2m+1)}(e_{f_0(2k+1)} - e_{f_0(2k)})(\xi^0_k)=0,
\]
which follows by $f_0(2m+1)<f_0(2k)<f_0(2k+1)$.
The second series also annihilates, indeed for $m>k$
we have $c^0_me_{f_0(2k+1)}= c^0_m e_{f_0(2m)}e_{f_0(2k+1)}=0$
(the same equation also holds for $e_{f_0(2k)}$). Using the same argument, it can
be proved that
\[
a_0(\xi)=c^0_n(\xi)
\]
for every $\xi \in S_{f_0(2n),f_0(2n+1)}$.
Observe that $\lVert (a_0 - e_{f_0(2m+1)}a_0)(\xi^0_k) \rVert =1$ for $k>m$, thus
$a_0 \notin A$.

The construction proceeds by transfinite induction on $\aleph_1$, the first uncountable cardinal.
At step $\beta < \aleph_1$ we assume to have a sequence of elements
$(a_\alpha)_{\alpha<\beta}$ in $M(A)_+$ and functions $(f_\alpha)_{\alpha < \beta}$ such that:
\begin{enumerate}[(i)]
\item \label{clause1} For all $\alpha< \beta$ the function $f_\alpha :\N \to \N$ is strictly increasing and,
given any other $\gamma<\alpha$, for all $k\in \N$ there exists $N \in \N$ such that
for all $j>N$ and all $i \in \N$ the following holds
\[
\lvert f_\alpha(j)- f_\gamma(i) \rvert > 2^k.
\]
Furthermore, we ask that for all $\alpha < \beta$ and all $k\in \N$:
\[
f_\alpha(2(k+1)) - f_\alpha(2k +1) > 2^{2k+1}.
\]
\item \label{clause2} For each $\alpha < \beta$ there exists a sequence $(c^\alpha_k )_{k \in \N}$
of positive norm 1 elements in $A$ such that
\[
a_\alpha = \sum_{k\in \N} c^\alpha_k .
\]
Moreover we require that
\[
e_{f_\alpha(2k+1)}c^\alpha_k=c^\alpha_ke_{f_\alpha(2k+1)}=c^\alpha_k,
\]
\[
e_{f_\alpha(2k)} c^\alpha_k = c^\alpha_k e_{f_\alpha(2k)}=0,
\]
and that there exist $\xi^\alpha_k , \eta^\alpha_k \in S_{f_\alpha(2k),f_\alpha(2k+1)}$, two norm one orthogonal vectors, such that $c^\alpha_k(\xi^\alpha_k)=\xi^\alpha_k$ and $c^\alpha_k(\eta^\alpha_k)=0$.
\item \label{clause3} Given $\alpha < \beta$ and $d_1, d_2 \in M(A)$, for all $l\in \N$,
and for all but possibly $l$ many $\gamma<\alpha$ the following holds:
\[
\lVert [(a_\alpha + d_1e_l),(a_\gamma + d_2e_l)] \rVert \ge \frac{1}{2}.
\]
\end{enumerate}
It can be shown, as we already did for $a_0$, that for all $\alpha < \beta$:
\begin{enumerate}[a.]
\item $a_\alpha\in M(A)_+ \setminus A$;
\item $\lVert a_\alpha  \rVert = 1$;
\item \label{itema}$a_\alpha(\xi)=c^\alpha_k(\xi) \in S_{f_\alpha(2k),f_\alpha(2k+1)}$ for every
$\xi \in S_{f_\alpha(2k),f_\alpha(2k+1)}$.
\end{enumerate}
Moreover, by items (\ref{clause1})--(\ref{clause2}), along with the fact that for $n_2<n_1<m_2<m_1$
\[
(e_{m_1}-e_{m_2})(e_{n_1}-e_{n_2})=0,
\]
we have that $a_\alpha a_\gamma \in A$ for all $\alpha, \gamma < \beta$.

We want to find $f_{\beta}$ and $a_{\beta}$ such that the families $\set{a_\alpha}_{\alpha< \beta +1}$
and $\set{f_{\alpha}}_{\alpha < \beta +1}$ satisfy the three inductive hypotheses.
This will be sufficient to continue the induction and to obtain
the thesis of the lemma. Since $\beta$ is a countable ordinal, the sequence
$(a_\alpha)_{\alpha<\beta}$ is either finite or can be written as $(a_{\alpha_n})_{n<\N}$,
where $n \mapsto \alpha_n$ is a bijection between $\N$ and $\beta$. We assume that
$\beta$ is infinite, since the finite case is easier. In order to ease the notation, we shall
denote $a_{\alpha_n}$ by $a_n$ (and similarly $f_{\alpha_n}$ by $f_n$,
$c^k_{\alpha_n}$ by $c^k_n$, etc.).

The construction of $a_\beta$ proceeds inductively on the set $\set{(i,j) \in \N \times \N : i\le j}$
ordered along with any well-ordering of type $\omega$ such that $(i,j) \le (i',j')$ implies $j \le j'$, like
for example
\[
(i,j) \le (i',j') \iff j \le j' \ \text{or} \ j = j', i \le i'.
\]
Suppose we are at step $M$, which corresponds to
a certain couple $(i,j)$. At step $M$ we
provide a $c^{\beta}_{M}\in A_+$ such that, for every $d_1, d_2 \in M(A)$
\[
\lVert 
[(a_j + d_1e_i),(c^{\beta}_{M} + d_2e_i)]
 \rVert \ge \frac{1}{2}
\]
and we define two values of $f_{\beta}$. Assume that $f_{\beta}(n)$
has been defined for $n \le 2M-1$. Let $m\in \N$ be the smallest natural number such that
\[
f_j(2m)>\max\left\{ i+2,f_{\beta}(2M-1)+2^{2M-1}+1 \right\}
\]
and such that, for $l\ge 2m$, the inequality $\lvert f_j(l) - f_k(n) \rvert > 2^{M} +1$
holds for all $k \in \N$ such that $\alpha_k < \alpha_j$, and all $n\in \N$.
By inductive hypothesis
there are two norm one orthogonal vectors $\xi^j_m,\eta^j_m \in
S_{f_j(2m),f_j(2m+1)}$ such that $c^j_m(\xi^j_m)=\xi^j_m$ and $c^j_m(\eta^j_m)=0$.
Set $\xi^\beta_{M}=\frac{1}{\sqrt{2}} (\xi_j^m + \eta_j^m)$ and $\eta^\beta_{M}=
\frac{1}{\sqrt{2}} (\xi_j^m - \eta_j^m)$.
Using Kadison's Transitivity Theorem, fix a positive, norm one element
\[c^{\beta}_{M} \in 
\overline{(e_{f_j(2m+1)}-e_{f_j(2m)})A(e_{f_j(2m+1)}-e_{f_j(2m)})}
\]
such that
\[
c^{\beta}_{M}(\xi^\beta_{M})= \xi^\beta_{M},
\]
\[
c^{\beta}_{M}(\eta^\beta_{M})=0.
\]
Let $f_{\beta}(2M) = f_j(2m)-1$ and $f_{\beta}(2M+1)=
f_j(2m+1)+1$. We have therefore that
\[
e_{f_{\beta}(2M+1)}
c^{\beta}_{M}e_{f_{\beta}(2M+1)}= c^{\beta}_{M},
\]
\[
e_{f_{\beta}(2M)} c^{\beta}_{M} = c^{\beta}_{M} e_{f_{\beta}(2M)}= 0.
\]
Moreover:
\[
\tag{$*$}
\lVert (a_j + d_1e_i)(c^{\beta}_{M} +d_2e_i)(\xi^\beta_{M})-
(c^{\beta}_{M} +d_2e_i)(a_j + d_1e_i)(\xi^\beta_{M}) \rVert =
\]
\[
\lVert a_j c^{\beta}_{M} (\xi^\beta_{M}) - c^{\beta}_{M} a_j (\xi^\beta_{M}) \rVert =
\frac{1}{2\sqrt{2}} \lVert  \xi_j^m - \eta_j^m \rVert  =
\frac{1}{2}.
\]
This is the case since $e_i(\xi)=0$
for every $\xi \in S_{f_j(2m),f_j(2m+1)}$ (we chose $m$ so that $f_j(2m) > i+2$) and
$c^{\beta}_{M} (\xi^\beta_{M}), a_j(\xi^\beta_{M}) = c^j_m (\xi^\beta_{M})
\in S_{f_j(2m),f_j(2m+1)}$.
Define
\[
a_{\beta} = \sum_{n \in \N} c^{\beta}_n =
\sum_{n \in \N} (e_{f_{\beta}(2n+1)}-e_{f_{\beta}(2n)})^{\frac{1}{2}} c^{\beta}_n
(e_{f_{\beta}(2n+1)}-e_{f_{\beta}(2n)})^{\frac{1}{2}}.
\]
This series is strictly convergent since all $c^\beta_n$'s have norm 1.
The families $\set{f_n}_{n < \N} \cup \set{f_\beta}$ and $\set{a_n}_{n < \N} \cup \set{a_\beta}$
satisfy items \eqref{clause1}--\eqref{clause2}
of the inductive hypothesis\footnote{The induction to define $a_\beta$ and $f_\beta$ is on the set
$\set{(i,j) \in \N \times \N : i\le j}$
ordered with a well-ordering of type $\omega$ such that $(i,j) \le (i',j')$ implies $j \le j'$.
This is used to show that $f_\beta$ satisfies clause \eqref{clause1} of the inductive hypothesis.}.

Finally we verify clause \eqref{clause3}.
Notice that, by construction, for every $k \in \N$, given $\xi \in S_{f_\beta(2k),f_\beta(2k+1)}$
we have
\[
a_\beta(\xi) = c^\beta_k(\xi).
\]
Let $i\le j \in \N$, denote the step corresponding to the couple $(i,j)$ by $M$,
and let $m\in \N$ be such that $f_{\beta}(2M)=f_j(2m)-1$
(by construction we can find such $m$).
Remember that $\xi^\beta_{M}=\frac{1}{\sqrt{2}} (\xi_j^m + \eta_j^m)
\in S_{f_{\beta}(2M),f_{\beta}(2M+1)}$. Given $d_1, d_2 \in M(A)$, we have that
\[
\lVert (a_j + d_1e_i)(a_{\beta} +d_2e_i)(\xi^\beta_{M})-
(a_{\beta} +d_2e_i)(a_j + d_1e_i)(\xi^\beta_{M}) \rVert =
\]
\[
\lVert a_j a_\beta(\xi^\beta_M) - a_\beta a_j(\xi^\beta_M) \rVert
= \frac{1}{2\sqrt{2}} \lVert  \xi_j^m - \eta_j^m \rVert  =
\frac{1}{2}.
\]
This equation can be shown using the same arguments used to prove $(*)$.

Notice that if $\beta$ is finite, we only obtain a finite number of $c^{\beta}_n$, therefore
their sum (which is finite) does not belong to $M(A) \setminus A$.
In this case it is sufficient to add an infinite number of addends, as
we did for $a_0$. Suppose that $\beta$ is (the ordinal corresponding to) $N \in \N$,
then the previous construction defines $f_N$ only up until $2N +1$.
Let $f_N(2(N+1))$ be the smallest integer such that
\begin{itemize}
\item $f_N(2(N+1)) - f_N(2N +1) > 2^{2N+1}$;
\item $\lvert f_N(2(N+1)) - f_j(n) \rvert > 2^{2(N+1)}$ for all $j<N$;
and for all $n\in \N$.
\end{itemize}
Define $f_N(2(N+1)+1)= f_N(2(N+1))+3$ and continue
inductively the definition of $f_N$.
For each $n>N$ we can therefore,
as we did for $a_0$ using Kadison's Transitivity Theorem, find a positive element
\[
c^N_n\in \overline{
(e_{f_N(2n+1)-1}-e_{f_N(2n)+1}) A (e_{f_N(2n+1)-1}-e_{f_N(2n)+1})}
\]
which moves a norm one vector $\xi^N_n \in S_{f_N(2n),f_N(2n+1)}$
into itself, and another orthogonal norm one vector $\eta^N_n$ to zero.
If we define $a_N$ to be the sum of such $c^N_n$'s, it
is possible to show, using the same arguments
exposed when $\beta$ was assumed to be infinite, that the families $\set{f_n}_{n < \N} \cup \set{f_\beta}$ and
$\set{a_n}_{n<N+1}$ satisfy
items \eqref{clause1}--\eqref{clause3} of the inductive hypothesis.
\end{proof}

The proof of Theorem \ref{thrm} is analogous to the one given in Theorem $7$ of \cite{bice}, but it uses our Lemma \ref{lemma} instead of \cite[Lemma 6]{bice}.

\begin{proof}[Proof of Theorem \ref{thrm}]
Let $(e_n)_{n\in \N} \subseteq A$ be the approximate unit defined at the beginning of the current section,
and let $(a_{\beta})_{\beta \in \aleph_1}$
be the $\aleph_1$-sized collection obtained from Lemma \ref{lemma}.
Suppose there is an uncountable $U\subseteq \aleph_1$ and $(d_{\beta})_{\beta \in U}
\subseteq A$ such that
\[
[(a_{\alpha} + d_{\alpha}),(a_{\beta} + d_{\beta})]=0
\]
for all $\alpha, \beta \in U$. By using the pigeonhole principle,
we can suppose that $\lVert d_{\beta} \rVert \le M$ for some $M \in
\mathbb{R}$, and that
there is a unique $n\in \N$ such that $\lVert d_{\beta} - d_{\beta}e_n \rVert \le \frac{1}{64(M+1)}$
for all $\beta \in U$.

Therefore, for every $\beta \in U$ and all but finitely many $\alpha \in U$ such that $\alpha< \beta$,
we have
\[
0 = \lVert [(a_{\alpha} + d_{\alpha}),(a_{\beta} + d_{\beta})] \rVert \ge
\lVert 
[(a_{\alpha} + d_{\alpha}e_n),(a_{\beta} + d_{\beta}e_n)]
 \rVert - \frac{1}{16} \ge \frac{1}{16}.
\]
This is a contradiction when $\set{\alpha \in U: \alpha < \beta}$ is infinite.
\end{proof}

\begin{proof}[Proof of Corollary \ref{crlry}]
The proof follows verbatim the one given for Lemma \ref{lemma} and Theorem \ref{thrm}.
The only difference is that each time Kadison's Transitivity Theorem is invoked in Lemma \ref{lemma}, it is possible
to use a stronger version of such theorem for \cstar-algebras with real rank zero (see for instance
Theorem $6.5$ of \cite{proj}), which allows us to choose a projection at each step. This stronger
version of Kadison's Transitivity Theorem can be used throughout the whole iteration since
hereditary subalgebras of real rank zero \cstar-algebras have real rank zero.
\end{proof}

\section{Concluding remarks and questions} \label{sctn:3}
If $A$ is a commutative non-unital \cstar-algebra, then the problem of
lifting commuting elements from $Q(A)$ to $M(A)$ is trivial, as both $M(A)$ and $Q(A)$ are
abelian. In Section \ref{sctn:2} we ruled out this possibility by
asking for $A$ to be primitive.

The other important feature we required to prove Theorem \ref{thrm} is $\sigma$-unitality.
We do not know whether this assumption could be weakened,
but it certainly cannot be removed tout-court.
Indeed, there are extreme examples of primitive,
non-$\sigma$-unital
\cstar-algebras whose corona is finite-dimensional (see \cite{innder} and
\cite{ghasemi2016extension}), for which Theorem \ref{thrm} is trivially false.
Our conjecture is that there might be a condition on the order structure of the
approximate unit of $A$ which is weaker than $\sigma$-unitality, but still
makes Theorem \ref{thrm} true. For instance, it would be interesting to know whether
the techniques used in Theorem \ref{thrm} could be applied to the algebra of
the compact operators on a nonseparable Hilbert space, or more in general to a \cstar-algebra
$A$ with a projection $p \in M(A)$ such that $pAp$ is primitive, non-unital
and $\sigma$-unital.

We remark that the proof of
Theorem \ref{thrm} we gave can be adapted
to any primitive \cstar-algebra $A$ which admits an
increasing approximate unit $\set{e_{\alpha}}_{\alpha
\in \kappa}$, for $\kappa$ regular cardinal, to produce a $\kappa^+$-sized
family of orthogonal positive elements in $Q(A)$ which cannot be lifted to a set
of commuting elements in $M(A)$.

Another lifting problem that could be considered, is the following.
\begin{question} \label{quest}
Assume $F\subseteq Q(A)_{sa}$ is a commutative family
such that any smaller (in the sense of cardinality)
subset can be lifted to a set of commuting elements in $M(A)_{sa}$.
Can $F$ be lifted
to a collection of commuting elements in $M(A)_{sa}$?
\end{question}
Theorem \ref{thrm} and
Proposition \ref{prop} entail that this is not true in general for primitive, non-unital, $\sigma$-unital
\cstar-algebras if $\lvert F \rvert = \aleph_1$,
pointing out the set theoretic incompactness of $\aleph_1$ for
this property.

If the family $F$ is infinite and countable, then Question \ref{quest} has a positive answer in
the Calkin algebra.
\begin{proposition} \label{propfin}
Suppose that $A$ is a separable abelian \cstar-subalgebra of $Q(H)$
such that every finitely-generated subalgebra of $A$ has an abelian lift. Then $A$ has an abelian lift.
\end{proposition}
The proof of this proposition relies on Voiculescu's Theorem
\cite[Theorem 3.4.6]{higsonroe}, starting from the following lemma. Given a map $\phi: A \to Q(H)$,
we say that $\Phi: A \to B(H)$ \emph{lifts} $\phi$ if $\phi = \pi \circ \Phi$, where $\pi: B(H) \to Q(H)$
is the quotient map.
\begin{lemma} \label{lemmafin}
Let $A$ be a separable unital abelian \cstar-subalgebra of $Q(H)$.
If there exists a unital abelian \cstar-algebra $B \subseteq B(H)$ lifting $A$,
then there is a unital $*$-homomorphism $\Phi: A \to B(H)$ lifting the
identity map on $A$.
\end{lemma}
\begin{proof}
Since $B$ is abelian, there exists a masa (maximal abelian subalgebra) of $B(H)$ containing $B$.
Masas in $B(H)$ are von Neumann algebras and, as such, they are generated by their
projections. This entails that $A$ is contained in a separable unital
abelian subalgebra $C(Y)$ of $Q(H)$ which is generated by its projections. By
\cite[Theorem 1.15]{bdf} there exists a unital $*$-homomorphism $\Psi: C(Y) \to B(H)$
lifting the identity on $C(Y)$. Let $\Phi$ be the restriction of $\Psi$ to $C(X)$.
\end{proof}
\begin{proof}[Proof of Proposition \ref{propfin}]
Suppose that $F = \set{a_n}_{n \in \N} \subseteq Q(H)_{sa}$ is an abelian family such
that every finite subset of $F$ has a commutative lift. Without loss of generality, we can assume
that $a_0 = 1$.
By Lemma \ref{lemmafin} we can assume
that, for every $k \in \N$, there is a unital $*$-homomorphism
$\Phi_k: \mathrm{C}^\ast(\set{a_n}_{n \le k}) \to B(H)$ lifting the identity map on
$\mathrm{C}^\ast(\set{a_n}_{n \le k})$. By Voiculescu's Theorem \cite[Theorem 3.4.6]{higsonroe} we
can moreover assume that, for every $n \in \N$, the sequence $\set{\Phi_k(a_n)}_{k \ge n}$
converges to some self-adjoint operator $A_n$ in $B(H)$ such that $A_n - \Phi_k(a_n)$
is compact for every $k \in \N$. The family $\set{A_n}_{n \in \N}$ is a
commutative lifting of $\set{a_n}_{n \in \N}$.
\end{proof}

More general forms of Voiculescu's Theorem are known to hold for extensions of various
separable \cstar-algebras other than $K(H)$ (see \cite{ellvoicu}, \cite{gabe}, \cite[Section 2.2]{chris}).
Such generalizations could potentially be used to carry out the arguments exposed above for coronas
of other separable nuclear stable \cstar-algebras. We remark however the importance
of being able to lift separable abelian subalgebras of $Q(H)$ to abelian algebras in $B(H)$ with the
same spectrum, as guaranteed by Lemma \ref{lemmafin}. This is false in general in other
coronas, as it happens for instance when $A = \mathcal{Z} \otimes K(H)$. In this case, projections
in $Q(A)$ do not necessarily lift to projections in $M(A)$, since the former has real rank zero but the
latter has not (see \cite{linng}).

The following example shows that Question \ref{quest} has negative answer
for finite families with an even number of elements.
\begin{example}
Let $S^n$ be the $n$-dimensional
sphere. The algebra $C(S^n)$ is generated by $n+1$ self-adjoint elements $\set{h_i}_{0\le i \le n}$
satisfying the relation
\[
h_0^2 + \dots + h^2_n=1.
\]
Let $F=\set{h_i}_{0\le i \le n}$. The relation above implies that the joint
spectrum of a subset of $F$ of size $m\le n$
is the $m$-dimensional ball $B^m$. The space $B^m$ is contractible, therefore
the group $\text{Ext}(B^m)$ is trivial (see \cite[Section 2.6--2.7]{higsonroe}
for the definition of the functor $\text{Ext}$
and its basic properties). As a consequence, for any $[\tau] \in \text{Ext}(S^n)$,
any proper subset of $\tau[F]$ can be lifted to a set of commuting self-adjoint operators in
$B(H)$. On the other hand $\text{Ext}(S^{2k+1})
=\mathbb{Z}$ for every $k\in \N$. We conclude that
any non-trivial extension $\tau$ of $C(S^{2k+1})$ produces, by Lemma \ref{lemmafin}, a family
$\tau[F]$ of size $2k+2$ in the Calkin algebra for which Question \ref{quest} has negative answer.
\end{example}
The argument above does not apply to families of odd cardinality, since
$\text{Ext}(S^{2k})=\set{0}$ for every $k\in \N$. However, in \cite{ken2} (see also
\cite{voic}, \cite{lor3}), the author builds a set of three commuting self-adjoint
elements in the corona algebra of
$\bigoplus_{n\in \N} M_n(\C)$ with no commutative lifting to the multiplier algebra,
whose proper subsets of size two all admit a commutative lifting.
The answer to Question \ref{quest} for larger finite families with an odd number of elements is,
to the best of our knowledge, unknown.

\begin{acknowledgements}
My sincerest gratitude goes to Ilijas Farah, for his patience, and for
the impressive amount of crucial suggestions
he gave me while working on this problem and on the earlier drafts of this paper.
I wish to thank Alessandro Vignati for the stimulating conversations we had, for the knowledge
he shared with me, and for the suggestions on the earlier drafts of this paper.
I wish to thank George
Elliott for the incredible amount of interesting questions related to this topic he posed
and for having suggested me a more intuitive way to present this problem and
a more readable form of Lemma \ref{lemma}.
\end{acknowledgements}

\bibliographystyle{amsalpha}
	\bibliography{Bibliography}

\Addresses

\end{document}